\documentclass{siamltex}
\usepackage{graphicx}
\usepackage{bm}
\usepackage{amsmath}
\usepackage{amssymb}
\usepackage{epsfig}
\usepackage{latexsym}
\usepackage{amsbsy}
\usepackage{lineno}
\usepackage{xparse}
\usepackage[usenames, dvipsnames]{color}
\usepackage[breaklinks,colorlinks=true,allcolors=blue]{hyperref}

\NewDocumentCommand{\dgal}{sO{}m}{%
  \IfBooleanTF{#1}
    {\dgalext{#3}}
    {\dgalx[#2]{#3}}%
}
\NewDocumentCommand{\dgalext}{m}{%
  \sbox0{%
    \mathsurround=0pt 
    $\left\{\vphantom{#1}\right.\kern-\nulldelimiterspace$%
  }%
  \sbox2{\{}%
  \ifdim\ht0=\ht2
    \{\kern-.45\wd2 \{#1\}\kern-.45\wd2 \}%
  \else
    \left\{\kern-.5\wd0\left\{#1\right\}\kern-.5\wd0\right\}%
  \fi
}

\NewDocumentCommand{\dgalx}{om}{%
  \sbox0{\mathsurround=0pt$#1\{$}%
  \sbox2{\{}%
  \ifdim\ht0=\ht2
    \{\kern-.45\wd2 \{#2\}\kern-.45\wd2 \}%
  \else
    \mathopen{#1\{\kern-.5\wd0 #1\{}
    #2
    \mathclose{#1\}\kern-.5\wd0 #1\}}
  \fi
}

\newtheorem{remark}{Remark}[section]

\newtheorem{assumption}{Assumption}[section]

\newcommand{\sigmab}{\mbox{\boldmath$\sigma$}}

\font\msbm=msbm10

\newcommand{\R}{\hbox{{\msbm \char "52}}}

\definecolor{otherblue}{rgb}{0,0.3,0.6}
\def\rbl#1{{\textcolor{black}{#1}}}
\def\rrd#1{{\textcolor{black}{#1}}}

\def\rrb#1{{\textcolor{black}{#1}}}
\def\cpbl#1{{\textcolor{black}{#1}}}

\title{{Robust preconditioning for stochastic Galerkin formulations of parameter-dependent {\cpbl{nearly incompressible}} elasticity equations}\thanks{{This work was supported 
by  EPSRC grant EP/P013317. The authors would also like to thank the Isaac Newton Institute for Mathematical Sciences, Cambridge, for support and hospitality during the Uncertainty Quantification programme where work on this paper was partially undertaken. This programme was supported by EPSRC grant EP/K032208/1.}}}

\author{
Arbaz Khan\thanks{
School of Mathematics, University of Manchester, UK (\tt{arbaz.khan@manchester.ac.uk})}
\and
Catherine E. Powell\thanks{
School of Mathematics, University of Manchester, UK (\tt{c.powell@manchester.ac.uk})}
\and
David J. Silvester\thanks{
School of Mathematics, University of Manchester, UK (\tt{d.silvester@manchester.ac.uk})}.
}

\begin{document}

\maketitle

\begin{abstract}
We consider the nearly incompressible linear elasticity problem with an uncertain spatially varying Young's modulus. The uncertainty is modelled with a finite set of parameters
 with prescribed probability distribution.  We introduce a novel three-field mixed variational formulation of the PDE model and discuss its approximation by stochastic Galerkin mixed finite element techniques. First, we establish the well-posedness of the proposed variational formulation and the associated finite-dimensional approximation. Second, we focus on
the efficient solution of the associated large and indefinite linear system of equations.  A new preconditioner is introduced for use with the minimal residual method (MINRES). 
Eigenvalue bounds for the preconditioned system are established and shown to be independent of the discretisation parameters and the Poisson ratio. The S-IFISS software used for computation is available online.
\end{abstract}



\pagestyle{myheadings}

\thispagestyle{plain}
\markboth{A.\ KHAN, C.\  E.\ POWELL and D.\  J.\  SILVESTER}
{Preconditioning linear elasticity problems with \rbl{uncertain inputs}}

\section{Introduction}\label{sec11}
\rbl{The locking of finite element  approximations when solving nearly incompressible elasticity problems is a  significant  issue in the computational engineering world. 
The standard way of  preventing locking is to write the underlying equations as a system and \rrb{introduce}  pressure as an additional unknown~\cite{RLH, KPS}. Thus,
the starting point for this work is the  {\sl Herrmann} formulation of  linear elasticity}
\begin{subequations} 
\begin{align*}
 -\nabla\cdot\sigmab & =\bm{f} \quad  \textrm{in } \,  D, \\
 \nabla\cdot\bm{u}+\frac{p}{\lambda} &=0, \quad \textrm{in } D,
\end{align*}
\end{subequations}
where $D$ is a bounded Lipschitz polygon in $\R^2$ \rbl{(polyhedral in  $\R^3$)}. \rbl{In this setting,} 
the  \rbl{elastic}  deformation of \rbl{the} isotropic {solid} is defined in terms of the stress tensor $\sigmab$,
  the body force $\bm{f}$, the displacement field $\bm{u}$ and the Herrmann pressure 
  $p$ (auxiliary variable). 
  The stress tensor is related to the strain tensor $\bm{\varepsilon}$ through the identities
\begin{align*}
\sigmab=2 \mu \bm{\varepsilon}- p{\bm{I}}, \qquad  \bm{\varepsilon}=\frac{1}{2}\left(\nabla \bm{u}+(\nabla \bm{u})^{\top}\right).
\end{align*}
The Lam\'{e}  coefficients $\mu$ and $\lambda$ satisfy $0<\mu_1<\mu<\mu_2 <\infty$ 
and $0<\lambda<\infty$ and can be defined in 
terms of the Young's modulus $E$ and the Poisson ratio $\nu$ \rbl{via}
\begin{align*}
\mu=\frac{E}{2(1+\nu)}, \quad \lambda=\frac{E\nu}{(1+\nu)(1-2\nu)}.
\end{align*}
Our focus is on uncertainty quantification. Specifically, we consider the case where the properties of the elastic material are varying spatially in an uncertain way. \cpbl{For example, this may be due to material imperfections or inaccurate measurements}. To account for this uncertainty we model the Young's modulus $E$ as a spatially varying random field. \cpbl{More precisely, we introduce} a vector $\bm{y}=(y_1,\ldots,y_M)$ of parameters, with each $y_{k} \in \Gamma_{k}=[-1,1]$  and \cpbl{represent} $E$ as an affine combination
\begin{align}\label{E-def}
E(\bm{x},\bm{y}):= e_0(\bm{x})+\sum_{k=1}^{M}e_k(\bm{x}) y_{k}, \quad \bm{x}\in D, \bm{y} \in \Gamma,
\end{align}
where $\Gamma = \Gamma_{1} \times \cdots \times \Gamma_{M} \subset \mathbb{R}^{M}$ is 
\rbl{our} parameter domain. \cpbl{Such representations arise, for example, from truncated Karhunen--Lo\`eve expansions of second-order random fields.  In \eqref{E-def}, $e_{0}(\bm{x})$ typically represents the mean, and $e_{k}(\bm{x})y_{k}$ is a perturbation away from the mean}. The resulting parameter-dependent \rbl{problem is given by}
\begin{subequations} \label{os1a}
\begin{align}
 -\nabla\cdot\sigmab(\bm{x},\bm{y})& =\bm{f}(\bm{x}), \quad\bm{x} \in D, \;\bm{y}\in\Gamma, \\
 \nabla\cdot\bm{u}(\bm{x},\bm{y})+\frac{p(\bm{x},\bm{y})}{\lambda(\bm{x},\bm{y})} &=0, \quad \, \,  \, \quad\bm{x}\in D, \;\bm{y}\in\Gamma,\\
 \bm{u}(\bm{x},\bm{y})&=\bm{g}(\bm{x}), \quad \bm{x}  \in  \partial D_D,\;\bm{y}\in\Gamma, \\
 \sigmab(\bm{x},\bm{y}) \bm{n}&= {\bm{0}}, \quad \, \quad \,  \bm{x} \in\rbl{ \partial D_N},\;\bm{y}\in\Gamma, \label{os2b}
\end{align}
\end{subequations}
\rbl{where} the boundary of the spatial domain  \rbl{is} $\partial D = \partial D_D \cup \partial D_N$
\rbl{with} $\partial D_D \cap \partial D_N= \emptyset$ and $\partial D_D, \partial D_N\neq \emptyset$,
\rbl{the stress tensor is $\sigmab : D\times\Gamma\rightarrow \rrb{\R^{d\times d}\, (d=2,3)}$, the strain tensor is
$\bm{\varepsilon} : D\times\Gamma\rightarrow \rrb{\R^{d\times d}}$, the body force is $\bm{f}: D\rightarrow \rrb{\R^{d}}$, 
the displacement field is $\bm{u}: D\times\Gamma\rightarrow \rrb{\R^{d}} $ and the Herrmann pressure is
$p: D\times\Gamma\rightarrow \R$}. 
The Lam\'{e}  coefficients are also parameter-dependent and spatially varying
\begin{align*}
\mu(\bm{x},\bm{y})=\frac{E(\bm{x},\bm{y})}{2(1+\nu)}, \quad \lambda(\bm{x},\bm{y})=\frac{E(\bm{x},\bm{y})\nu}{(1+\nu)(1-2\nu)}.
\end{align*}
Note that we assume that $\nu$ is a given fixed constant and that $0<\mu_1<\mu<\mu_2 <\infty$ and $0<\lambda<\infty$ a.e. in $D \times \Gamma$.

Stochastic Galerkin finite element methods (\rrd{SGFEMs})  \rrb{are} a popular 
\rbl{way of} approximating solutions to parameter-dependent PDEs. Broadly speaking, we seek approximate solutions in tensor product 
spaces of the form $X_{h} \otimes S_{\Lambda}$ where $X_{h}$ is an appropriate finite element space 
associated with a {subdivision of} $D$ and 
$S_{\Lambda}$ is, \rbl{typically,} a set of multivariate polynomials that are globally defined on 
the parameter domain $\Gamma$.  {\rrd{This is a feasible strategy if the number of input parameters is modest, and the underlying solution is sufficiently smooth as a function of those parameters.}} 
\rrb{ We refer to Babu{\v s}ka et al.~\cite{babuska2004galerkin} for a priori error estimates for SGFEM approximations of solutions to elliptic PDEs with parameter-dependent coefficients} 
\rrb{and  Bespalov et al.~\cite{bespalov2012priori} for a priori error estimates for SGFEM approximations of solutions to mixed formulations of elliptic PDEs with parameter-dependent coefficients.}
\rrb{A posteriori error analysis of linear elasticity with parameter-dependent coefficients is considered by Eigel et al. \cite{eigel2014adaptive}.} 
Crucial to the efficient implementation of SGFEMs is the need \rbl{to 
separate the} terms that depend on $\bm{x}$ \rbl{from the} terms that depend on $\bm{y}$ in the weak 
formulation of the problem. Here, since both $\mu$ and $\rbl{1/\lambda}$ appear in the 
PDE model (\ref{os1a}), both $E$ and $E^{-1}$ appear in the formulation. 

To address this difficulty, our idea here is \cpbl{to} introduce a second auxiliary variable $\tilde{p}=p/E$ to give a distinctive  three-field mixed formulation of  (\ref{os1a}):  find $\bm{u}: D\times\Gamma \to \rrb{\mathbb{R}^{d}}$ and ${{p}}, \tilde{p}:D\times\Gamma \to \mathbb{R}$ such that, 
\begin{subequations}\label{os3}
\begin{align}\label{3fieldfor11}
 -\nabla\cdot\sigmab(\bm{x},\bm{y})& =\bm{f}(\bm{x}), \quad \bm{x}\in D, \;\bm{y}\in\Gamma, \\
 \nabla\cdot\bm{u}(\bm{x},\bm{y})+ \tilde{\lambda}^{-1} \tilde{p}(\bm{x},\bm{y}) &=0,\quad \quad \, \, \, \bm{x}\in D, \;\bm{y}\in\Gamma, \label{3fieldfor12}\\
\tilde{\lambda}^{-1} {p}(\bm{x},\bm{y}) - \tilde{\lambda}^{-1} E(\bm{x},\bm{y}) \tilde{p}(\bm{x},\bm{y}) &=0,\quad \quad \, \, \, \bm{x}\in D, \;\bm{y}\in\Gamma, \label{3fieldfor13}\\
 \bm{u}(\bm{x},\bm{y})&=\bm{g}(\bm{x}),\quad \bm{x} \in \partial D_D,\;\bm{y}\in\Gamma,\\
 \sigmab(\bm{x},\bm{y}) \bm{n}&= \cpbl{\bm{0}},\quad \quad \, \, \, \bm{x} \in \rbl{\partial D_N},\;\bm{y}\in\Gamma ,
\end{align}      
\end{subequations}
where 
\begin{align*}
\tilde{\lambda}= {\frac{\lambda(\bm{x},\bm{y})}{E(\bm{x},\bm{y})}}= {\frac{\nu}{(1+\nu)(1-2\nu)}},
\end{align*}
is now a fixed constant. The advantage of 
(\ref{os3}) is that while $E$ appears in the first and third equations, $E^{-1}$ 
does not appear at all. As a result, the discrete problem associated with our SGFEM 
approximation  \rbl{has a structure that is relatively easy to exploit. 
This is a novel solution strategy and gives this work a distinctive edge.}

The rest of the paper is organised as follows. 
Section~\ref{EPWRD} \rbl{introduces a weak formulation of \eqref{os3} 
and discusses {well posedness}}. In particular, a stability result is established with respect to a 
coefficient-dependent  norm  \rbl{that is a generalisation of  the natural norm  
 identified in our earlier work~\cite{KPS}. 
Section \ref{disformul} introduces} the finite-dimensional problem associated with SGFEM 
approximation and gives details of the associated \rbl{linear algebra system that needs to be solved
when computing the Galerkin solution}. A novel preconditioner is introduced in Section \ref{precond-sec} and 
bounds for the eigenvalues of the preconditioned system are \rrd{established}. 
The preconditioning  strategy is consistent with  the philosophy of  Mardal and
Winther~\cite{mardal2011preconditioning}:  the diagonal blocks of the 
preconditioning matrix are associated \rbl{with the norm for which the 
stability of the mixed approximation has been established.}
Finally, we present numerical results in Section~\ref{Numer} to 
illustrate the efficiency and robustness  when \rbl{representative discrete problems are  solved using
the minimal residual method}. 

\section{Weak formulation}\label{EPWRD}
\cpbl{First,} we need to impose some conditions on the model inputs and define appropriate solution spaces. Recall that \cpbl{$E$ is defined as in \eqref{E-def}}
where \cpbl{$\Gamma=\Gamma_{1} \times \cdots \times \Gamma_{M} \subset \mathbb{R}^{M}$} is the parameter domain, and $\Gamma_k=[-1,1]$. 

\begin{assumption}\label{Assump2}
\rbl{The random field} $E\in L^{\infty}(D\times \Gamma)$ is uniformly bounded away from zero, i.e., there exist positive constants $E_{\min}$ and $E_{\max}$ such that
\begin{align}\label{bounde11}
0<E_{\min}\le E(\bm{x},\bm{y}) \le E_{\max} <\infty \quad \mbox{\rm{a.e. in}}\, D\times \Gamma.
\end{align}
\rbl{To identify the lower bound, it will be convenient to further assume that} 
\begin{align}
\label{bounde11x}
\cpbl{0<e_0^{\min}\le e_0(\bm{x})\le e_0^{\max} < \infty  \quad \mbox{\rm{a.e. in}}\, D \quad \mbox{and}  \quad \frac{1}{e_0^{\min}}\sum_{k=1}^{M}|| e_{k} ||_{L^{\infty}(D)}<1}.
\end{align}
\end{assumption}

Let $\pi(\bm{y})$ be a product measure with $\pi(\bm{y}):= \Pi_{k=1}^{M}\cpbl{\pi_k(y_{k})}$, where $\pi_k$ denotes a measure on $(\Gamma_k, \mathcal{B}(\Gamma_k))$ and $\mathcal{B}(\Gamma_k)$ is the Borel $\sigma$-algebra on $\Gamma_k$. We will assume that the parameters $y_{k}$ in \eqref{E-def} are images of independent mean zero uniform random variables on $[-1,1]$ and choose $\pi_{k}$ to be the associated probability measure. Now we can define the following Bochner space, 
\begin{align*}
L^2_{\pi}(\Gamma, X(D)) :=\left \{v(\bm{x},\bm{y}) : D\times\Gamma\rightarrow \mathbb{R}; ||v||_{L^2_{\pi}(\Gamma, X(D))} <\infty\right\},
\end{align*}
where $X(D)$ is a normed vector space of real-valued functions on $D$ with norm $||\cdot||_X$ and 
\begin{align}
||\cdot||_{L^2_\pi(\Gamma,X(D))} := \left(\int_{\Gamma}||\cdot||_X^2 d\pi(\bm{y})\right)^{1/2}.
\end{align} 
In our analysis, we will need the following spaces
\begin{align*}
\bm{\mathcal{V}} := \cpbl{L^2_{\pi}(\Gamma, \bm{H}^1_{E_0}(D))},\quad \mathcal{W} := \cpbl{L^2_{\pi}(\Gamma, {L}^2(D))}  \quad \mbox{and} \quad\bm{\mathcal{W}} := \cpbl{L^2_{\pi}(\Gamma, \bm{L}^2(D))},
\end{align*}
where $\bm{H}^1_{E_0}(D)= \{ \bm{v} \in\bm{H}^1(D),  \bm{v} |_{\partial D_D}=\bm{0}\}$ 
and $\bm{H}^1(D)=\bm{H}^1(D;\mathbb{R}^{\rbl{d}})$ is \cpbl{the usual} vector-valued Sobolev space 
with associated norm $||\cdot||_{1}$. We assume that the load function satisfies $\bm{f}\in (\cpbl{L^{2}(D)})^{\rbl{d}}$ and for simplicity, 
we choose $\bm{g}= \bm{0}$ on $\partial D_D$. In that case, the weak formulation of 
\eqref{os3} is: find $(\bm{u},p,\tilde{p})\in \bm{\mathcal{V}}\times\mathcal{W}\times\mathcal{W}$ such that
\begin{subequations} \label{scm11a}
\begin{align}
a(\bm{u},\bm{v})+b(\bm{v},p)&=f(\bm{v}) \quad\forall \bm{v}\in \bm{\mathcal{V}},\\
b(\bm{u},q)-c(\tilde{p},q)&=0 \quad \quad \, \, \forall q\in \mathcal{W},\\
-c(p,\tilde{q})+d(\tilde{p},\tilde{q})&=0 \, \, \quad \quad \forall \tilde{q}\in \mathcal{W}.
\end{align}
\end{subequations}
Here, we have
\begin{align}\label{fbf}
a(\bm{u},\bm{v})&:= \cpbl{\alpha} \int_{\Gamma}\int_{D}  E(\bm{x},\bm{y}) \bm{\varepsilon}(\bm{u}(\bm{x},\bm{y})):\bm{\varepsilon}(\bm{v}(\bm{x},\bm{y}))d\bm{x}d\pi(\bm{y}),\\
b(\bm{v},p)&:=-\int_{\Gamma}\int_{D}  {p}(\bm{x},\bm{y}){\rm{div}}\,\bm{v}(\bm{x},\bm{y})d\bm{x}d\pi(\bm{y}),  \\
c(p,q)&:=(\rbl{\alpha}\beta)^{-1}\int_{\Gamma}\int_{D} {p}(\bm{x},\bm{y}){q}(\bm{x},\bm{y})d\bm{x}d\pi(\bm{y}),\\
d(p,q)&:=(\rbl{\alpha}\beta)^{-1}\int_{\Gamma}\int_{D} E(\bm{x},\bm{y}) {p}(\bm{x},\bm{y}){q}(\bm{x},\bm{y})d\bm{x}d\pi(\bm{y}), \\
f(\bm{v})& : = \int_{\Gamma}\int_{D} f(\bm{x}) \bm{v}(\bm{x},\bm{y})d\bm{x}d\pi(\bm{y}),
\end{align}
with 
\begin{align}
\alpha := \rbl{1\over 1+ \nu}, \qquad \beta := \frac{\nu}{(1-2\nu)}.
\end{align}
Note that $\alpha$ and $\beta$ depend on the Poisson ratio $\nu$ but are fixed constants. 
Following convention, we will also define the bilinear form
\begin{align}\label{big_B}
\mathcal{B}(\bm{u},p,\tilde{p}; \bm{v},q,\tilde{q})=a(\bm{u},\bm{v})+b(\bm{v},p)+b(\bm{u},q)-c(\tilde{p},q)-c(p,\tilde{q})+d(\tilde{p},\tilde{q}),
\end{align}
so as to express (\ref{scm11a}) in the compact form: find $(\bm{u},p,\tilde{p})\in \bm{\mathcal{V}}\times\mathcal{W}\times\mathcal{W}$ such that 
\begin{align}\label{scm12}
\mathcal{B}(\bm{u},p,\tilde{p}; \bm{v},q,\tilde{q})=f(\bm{v}), \quad \forall (\bm{v},q,\tilde{q})\in \bm{\mathcal{V}}\times\mathcal{W}\times\mathcal{W}.
\end{align}

The next result establishes that the four bilinear forms appearing in \eqref{scm11a}  \cpbl{and hence the bilinear form $\mathcal{B}(\cdot, \cdot)$ in \eqref{scm12}, are bounded}. 
\begin{lemma}\label{tecre11}
 If $E$ satisfies Assumption \ref{Assump2}, then the following bounds hold
\begin{align}
a(\bm{u},\bm{v}) & \le \rbl{\alpha} {E}_{\max}||\nabla\bm{u}||_{\bm{\mathcal{W}}} 
||\nabla\bm{v}||_{\bm{\mathcal{W}}} \quad\forall \bm{u}, \bm{v} \in \bm{\mathcal{V}}, \\
b(\bm{u}, p)&\le \rbl{\sqrt{d}} \, ||\nabla\bm{u}||_{\bm{\mathcal{W}}}||p||_{\mathcal{W}}
\quad \, \qquad \, \, \forall \bm{u} \in \bm{\mathcal{V}}, \, \forall p \in \mathcal{W}, \\
c(p,q)&\le (\rbl{\alpha}\beta)^{-1} ||p||_{\mathcal{W}} ||q||_{\mathcal{W}} \quad \quad \quad\forall p,q\in\mathcal{W},\\
d(p,q)&\le (\rbl{\alpha}\beta)^{-1} {E}_{\max} ||p||_{\mathcal{W}} ||q||_{\mathcal{W}}\quad\forall p,q\in\mathcal{W}.
\label{dbound}
\end{align}
\end{lemma}
\begin{proof}
\rrd{All bounds follow from the   Cauchy--Schwarz  inequality and  (\ref{bounde11}).}
\end{proof}

The next result establishes that three of the bilinear forms appearing in \eqref{scm11a}  and \eqref{scm12} are coercive and that an inf--sup condition involving $b(\cdot, \cdot)$ is \cpbl{satisfied}. 
\begin{lemma}\label{tecre12}
If \rbl{ Assumption \ref{Assump2} is valid} then the following bounds hold
\begin{align}
a(\bm{u},\bm{u})& \ge \rbl{\alpha} {E}_{\min} 
C_{K}||\nabla\bm{u}||_{\bm{\mathcal{W}}}^2\quad\forall \bm{u}\in \bm{\mathcal{V}}, 
\label{abound}\\
c(p,p)&\ge  (\rbl{\alpha}\beta)^{-1} ||p||^2_{\mathcal{W}} \quad \quad \, \quad\forall  p\in \mathcal{W},
\label{cbound} \\
d(p,p)&\ge  (\rbl{\alpha}\beta)^{-1}{E}_{\min} ||p||^2_{\mathcal{W}}\quad\forall  p\in \mathcal{W},
\label{dubound}
\end{align}
where $\rbl{0<C_K \leq 1}$ is the Korn constant. In addition, 
there exists \rbl{an inf--sup constant} $C_{D}>0$ such that
\begin{align}\label{first_inf_sup}
\sup_{0\neq\bm{v}\in {\bm{\mathcal{V}}}}
\frac{b(\bm{v},q)}{ ||\nabla\bm{v}||_{\bm{\mathcal{W}}} }&\ge C_{D} 
 ||q||_{\mathcal{W}} \quad\forall  q\in \mathcal{W}.
 \end{align}
\end{lemma}
\begin{proof}
The first bound  follows by \rbl{combining  (\ref{bounde11}) with} Korn's inequality. The second and third bounds follow directly from the definition of the bilinear forms and (\ref{bounde11}).  \cpbl{We can use arguments such as in \cite[Lemma 11.2.3]{brenner_scott}, \cite[Lemma 7.2]{bespalov2012priori}  (and references therein for non-convex domains $D$)  as well as  \cite[Section 4.1.4]{ern} (for more general boundary conditions) to show that} for any $q\in\mathcal{W}$ there exists a $\bm{w}\in \bm{\mathcal{V}}$ 
such that ${\rm{div}}\,\bm{w}=q$ and 
$ \rbl{C_D} ||\nabla \bm{w}||_{\bm{\mathcal{W}}}\le  ||q||_{\mathcal{W}}$, 
where $C_D$ is positive constant. Thus
\begin{align*}
\sup_{0\neq\bm{v}\in {\bm{\mathcal{V}}}}
\frac{b(\bm{v},q)}{ ||\nabla\bm{v}||_{\bm{\mathcal{W}}} }&\ge \frac{-b(\bm{w},q)}{ ||\nabla\bm{w}||_{\bm{\mathcal{W}}} }=\frac{||q||_{\mathcal{W}}^2}{ ||\nabla\bm{w}||_{\bm{\mathcal{W}}} }\ge C_{D}||q||_{\mathcal{W}}.
\end{align*}
\end{proof}

To establish that our \rbl{problem formulation} is {well posed}, we \rbl{now} introduce a 
coefficient-dependent norm $||| \cdot |||$ on $\bm{\mathcal{V}}\times\mathcal{W}\times\mathcal{W}$, 
defined by
\begin{align}\label{normden}
|||(\bm{v},q,\tilde{q})|||^2 := \rbl{\alpha}||\nabla\bm{v}||_{\bm{\mathcal{W}}}^2+(\rbl{\alpha}^{-1}+(\rbl{\alpha}\beta)^{-1})||q||_{\mathcal{W}}^2+(\rbl{\alpha}\beta)^{-1}||\tilde{q}||_{\mathcal{W}}^2.
\end{align}
The well-posedness of  \eqref{scm12} is addressed in the next two \cpbl{results}.
 \begin{lemma}\label{welpose11} 
 \rbl{If Assumption \ref{Assump2} is valid then} for any 
 $(\bm{u},p,\tilde{p})\in \bm{\mathcal{V}}\times\mathcal{W}\times\mathcal{W}$, there exists 
 $(\bm{v},q,\tilde{q})\in \bm{\mathcal{V}}\times\mathcal{W}\times\mathcal{W}$ 
 with $ |||(\bm{v},q,  \tilde{q} )|||\le C_2 |||(\bm{u},p,\tilde{p})|||$, satisfying
 \begin{align}\label{lower_S_bound}
\mathcal{B}(\bm{u},p,\tilde{p}; \bm{v},q,\tilde{q})\ge {E}_{\min}C_1 |||(\bm{u},p,  \tilde{p})|||^2,
 \end{align}
 where $C_1$ and $C_2$ \rbl{depend on} ${E}_{\max}$,  $C_K$  and $C_D$.
  \end{lemma}
 \begin{proof}
 From (\ref{scm12}) \rbl{we have}
 \begin{align*}
 \mathcal{B}(\bm{u},p,\tilde{p}; \bm{u},-p,\tilde{p})&=a(\bm{u},\bm{u})+b(\bm{u},p)+b(\bm{u},-p)
 -c(\tilde{p},-p)-c(p,\tilde{p})+d(\tilde{p},\tilde{p}),\\
&=a(\bm{u},\bm{u})+ d(\tilde{p},\tilde{p})  =:  |\bm{u}|_a^2+ |\tilde{p}|^2_{d}.
 \end{align*}
Now, as a consequence of  \eqref{first_inf_sup}, since $p\in \mathcal{W}$, there exists a $\bm{w}\in \bm{\mathcal{V}}$ such that
\begin{align}
- b(\bm{w},p) \ge C_D  \rbl{\alpha}^{-1} ||p||^2_{\mathcal{W}},\quad \rbl{\alpha}^{1/2}||\nabla \bm{w}||_{\bm{\mathcal{W}}}\le \rbl{\alpha}^{-1/2}||p||_{\mathcal{W}}.
\end{align}
Using this particular $\bm{w}$ in \eqref{big_B} and using Lemma \ref{tecre11}, it follows that
\begin{align*}
 \mathcal{B}(\bm{u},p,\tilde{p}; -\bm{w},0,0)&=\rbl{-b(\bm{w},p) -a(\bm{u},\bm{w}) } \\
 &\ge C_D \, \rbl{\alpha}^{-1} ||p||^2_{\mathcal{W}}-|\bm{u}|_{a}|\bm{w}|_{a} \\
 &\ge C_D\,  \rbl{\alpha}^{-1} ||p||^2_{\mathcal{W}} 
             -|\bm{u}|_{a} \, E_{\max}^{1/2} \rbl{\alpha}^{1/2} ||\nabla \bm{w}||_{\bm{\mathcal{W}}}\\
&\ge C_D \, \rbl{\alpha}^{-1} ||p||^2_{\mathcal{W}}
             -|\bm{u}|_{a} \, E_{\max}^{1/2}\rbl{\alpha}^{-1/2} ||p||_{\mathcal{W}}\\
&\ge C_D\, \rbl{\alpha}^{-1} ||p||^2_{\mathcal{W}}-\frac{\epsilon} {2}|\bm{u}|_{a}^2
-\frac{ \rbl{\alpha}^{-1} E_{\max} }{2\epsilon} ||p||_{\mathcal{W}}^2,
 \end{align*}
for any $\epsilon >0$. \rbl{From \eqref{big_B} and using  \eqref{cbound} and \eqref{dbound} gives}
 \begin{align*}
 \mathcal{B}(\bm{u},p,\tilde{p}; 0,0,-{p})&=c(p,p)-d(\tilde{p},p)\\
 &\ge  (\rbl{\alpha}\beta)^{-1}||p||^2_{\mathcal{W}}-|\tilde{p}|_{d}|{p}|_{d}\\
&\ge (\rbl{\alpha}\beta)^{-1}||p||^2_{\mathcal{W}}- |\tilde{p}|_{d} 
\, E_{\max}^{1/2} (\rbl{\alpha}\beta)^{-1/2}||p||_{\mathcal{W}}\\
&\ge (\rbl{\alpha}\beta)^{-1}||p||^2_{\mathcal{W}}-\frac{\epsilon_1} {2}|\tilde{p}|_{d}^2-\frac{(\rbl{\alpha}\beta)^{-1}E_{\max}}{2\epsilon_1}||p||_{\mathcal{W}}^2
 \end{align*}
for any $\epsilon_{1} >0$. 
We now introduce two parameters $\delta >0$ and $\delta' >0$.
\rbl{Combining these two bounds} gives  
 \begin{align}
 \mathcal{B}(\bm{u},p,\tilde{p}; &\bm{u}-\delta \bm{w},-p,\tilde{p}-\delta'  p)\nonumber\\
 &=\mathcal{B}(\bm{u},p,\tilde{p}; \bm{u},-p,\tilde{p})
 +\delta\, \mathcal{B}(\bm{u},p,\tilde{p}; -\bm{w},0,0)+\delta' \, \mathcal{B}(\bm{u},p,\tilde{p}; 0,0,-p)\nonumber\\
 &\ge |\bm{u}|_a^2+ |\tilde{p}|^2_{d}+\delta \left(\frac{1}{\rbl{\alpha}}
 \left(C_D-\frac{E_{\max}}{2\epsilon}\right)||p||^2_{\mathcal{W}}-\frac{\epsilon} {2}|\bm{u}|_{a}^2\right)\nonumber\\
 &\qquad +\delta' \left(\frac{1}{\rbl{\alpha}\beta}\left(1-\frac{E_{\max}}{2\epsilon_1}\right)
 ||p||^2_{\mathcal{W}}-\frac{\epsilon_1} {2}||\tilde{p}||_{d}^2\right)\nonumber\\
 &=\left(1-\frac{\delta\epsilon}{2}\right)|\bm{u}|_{{a}}^2
 +\left(\frac{\delta}{\rbl{\alpha}}\left(C_D-\frac{E_{\max}}{2\epsilon}\right)
 +\frac{\delta'}{\rbl{\alpha}\beta}\left(1-\frac{E_{\max}}{2\epsilon_1}\right) \right)
 ||p||^2_{\mathcal{W}}\nonumber\\
 &\qquad +\left(1-\frac{\delta' \epsilon_1}{2}\right)  |\tilde{p}|_{d}^{2}. \nonumber
 \end{align}
 \rbl{Next}, making the specific choice
$$\epsilon= \frac{E_{\max}}{C_D}, \quad \delta=\frac{1}{\epsilon}=\frac{C_D}{E_{\max}}, \quad \epsilon_1= {E_{\max}}, \quad \delta'=\frac{1}{\epsilon_1}=\frac{1}{E_{\max}},$$ 
\rbl{and using   \eqref{abound} and \eqref{dubound} gives}
\begin{align}
  \mathcal{B}(\bm{u},p,\tilde{p}; &\bm{u}-\delta \bm{w},-p,\tilde{p}-\delta' p)\nonumber\\
  &\ge \frac{1}{2}|\bm{u}|_{{a}}^2+ \frac{1}{2} \left(\frac{C_D^2}{\rbl{\alpha} E_{\max}}
  +\frac{1}{\rbl{\alpha} \beta E_{\max}}\right)||p||^2_{\mathcal{W}}+\frac{1}{2}|\tilde{p}|_{d}^2,\nonumber\\
  &\ge \frac{1}{2} C_{K} E_{\min} \rbl{\alpha}\, ||\nabla\bm{u}||_{\bm{\mathcal{W}}}^2 
  + \frac{1}{2 E_{\max}} \left(\frac{C_D^2}{\rbl{\alpha} }
  +\frac{1}{\rbl{\alpha} \beta}\right)||p||^2_{\mathcal{W}}  
  + \frac{1}{ 2 \rbl{\alpha} \beta} E_{\min} \,  ||\tilde{p}||^2_{\mathcal{W}}, \nonumber\\
 &\ge C  \left( \rbl{\alpha} 
 ||\nabla\bm{u}||_{\bm{\mathcal{W}}}^2 
 +\left(\frac{1}{\rbl{\alpha}}+\frac{1}{\rbl{\alpha}\beta} \right)||p||^2_{\mathcal{W}}
 +\frac{1}{\rbl{\alpha}\beta} ||\tilde{p}||^2_{\mathcal{W}}\right) 
  =: C \, ||| (\bm{u}, p , \tilde{p}) |||^{2},  \nonumber
 \end{align}
 where $C= \frac{1}{2}\min\{ \rbl{E_{\min}C_K}, \frac{C_D^2}{E_{\max}},\frac{1}{E_{\max}} \}$.  
  \rbl{Since $E_{\min} \le E_{\max}$ we have shown that \eqref{lower_S_bound} holds with}
   $\bm{v}:=\bm{u} - \delta \bm{w},$ $q:=-p$, $\tilde{q} := \tilde{p} - \delta' p$ with
 $C \geq E_{\min} C_{1}$  where
 $C_{1}=\frac{1}{2} \min\{\rbl{C_K}, \frac{C_D^2}{E_{\max}^2},\frac{1}{E_{\max}^2}\}$.
\rbl{To complete the proof, we note that} 
  \begin{align}
\rbl{\alpha} ||\nabla(\bm{u}-\delta\bm{w})||_{\bm{\mathcal{W}}}^2
&\le 2 \rbl{\alpha} ||\nabla\bm{u}||_{\bm{\mathcal{W}}}^2
+2 \delta^2 \rbl{\alpha}||\nabla\bm{w}||_{\bm{\mathcal{W}}}^2 
\le 2 \rbl{\alpha} ||\nabla\bm{u}||_{\bm{\mathcal{W}}}^2
+2 \delta^2 \rbl{\alpha}^{-1}||p||_{{\mathcal{W}}}^2. \nonumber
 \end{align}
\rbl{Similarly,}
 \begin{align*}
(\rbl{\alpha} \beta)^{-1}  ||\tilde{p}-\delta' p||_{\mathcal{W}}^2 \le  
2 (\rbl{\alpha} \beta)^{-1}  \| \tilde{p} \|_{\mathcal{W}}^{2}  
+  2  \delta'^{2} (\rbl{\alpha} \beta)^{-1}  \| p \|_{\mathcal{W}}^{2}.
 \end{align*}
 Using the definition of the norm $ ||| \cdot |||$ then leads to the upper bound
 \begin{align*} 
 |||(\bm{u}-\delta\bm{w},&-p,\tilde{p}-\delta' p)|||^2\nonumber\\
 &= \rbl{\alpha} ||\nabla(\bm{u}-\delta\bm{w})||^2_{\bm{\mathcal{W}}}
 +\left(\frac{1}{\rbl{\alpha}}+\frac{1}{\rbl{\alpha}\beta}\right)
 ||p||_{\mathcal{W}}^2+\frac{1}{\rbl{\alpha}\beta}||\tilde{p}-\delta' p||_{\mathcal{W}}^2 \nonumber\\
&\le  (2+2 \delta^2+2 \delta'^2)  \Big( \rbl{\alpha} ||\nabla\bm{u}||_{\bm{\mathcal{W}}}^2
+\Big(\frac{1}{\rbl{\alpha}}+\frac{1}{\rbl{\alpha}\beta}\Big)
||p||_{{\mathcal{W}}}^2+\frac{1}{\rbl{\alpha}\beta}||\tilde{p}||_{\mathcal{W}}^2\Big) \nonumber\\
& =  C_{2}^{2}  \,  |||(\bm{u},p,\tilde{p})|||^2,
 \end{align*}
 as required.
 \end{proof}

The following theorem is an immediate consequence.
\begin{theorem}\label{welpose12}
\rbl{Given that  $E$ satisfies condition (\ref{bounde11})  in Assumption \ref{Assump2} 
the three-field} formulation (\ref{scm12}) admits a unique solution 
$(\bm{u},p,\tilde{p})\in\bm{\mathcal{V}}\times\mathcal{W}\times\mathcal{W}$. Moreover, 
 \begin{align}
|||(\bm{u},p,\tilde{p})|||\le \frac{C_3}{E_{\min}} \rbl{\alpha}^{-1/2} ||\bm{f}||_{L^{2}(D)}
\label{hadamard}
 \end{align}
 where $C_3$ depends on $E_{\max}$,  $C_K$ and $C_D$.
 \end{theorem}
 \begin{proof}
 \rbl{Lemma~\ref{welpose11} ensures that}
 \begin{align}\label{weleq11}
 C_1E_{\min}|||(\bm{u},p,\tilde{p})|||^2\le  \mathcal{B}(\bm{u},p,\tilde{p}; \bm{v},q,\tilde{q})={f}(\bm{v})
 \end{align}
 where $(\bm{v}, q, \tilde{q})$ satisfies $|||(\bm{v},q,\tilde{q})|||\le C_2 |||(\bm{u},p,\tilde{p})|||$. 
 \rbl{Applying Cauchy--Schwarz  to the right-hand side then} gives
 \begin{align*}
 C_1 E_{\min} |||(\bm{u},p,\tilde{p})|||^2 & \le {\alpha}^{-1/2} 
 ||\bm{f}||_{L^{2}(D)} \, \rbl{\alpha}^{1/2} \cpbl{||\bm{v}||_{L^2(\Gamma, L^2(D))}} \\
  & \le \rbl{\alpha}^{-1/2} ||\bm{f}||_{L^{2}(D)} \, \rbl{L} |||(\bm{v},q,\tilde{q})||| \\
 & \le \rbl{\alpha}^{-1/2} ||\bm{f}||_{L^{2}(D)} \, \rbl{L}  C_2 \,  |||(\bm{u},p,\tilde{p})|||,
 \end{align*}
\rbl{where $L$ is the  Poincar{\' e}--Friedrichs  constant associated with $D$}. This implies (\ref{hadamard})
with  $C_3:= \rbl{L} C_2/C_1$. 
 \end{proof}

\section{\rbl{Finite-dimensional formulation}}\label{disformul}

To construct \rbl{an SGFEM approximation of (\ref{scm11a}) we need to 
introduce a conforming} finite element space 
\begin{align*}
V_{h} = \textrm{span} \left\{ \phi_{1}(\bm{x}), \phi_{2}(\bm{x}), 
       \ldots, \phi_{n_{u}}(\bm{x}) \right \} \subset H_{E_{0}}^{1}(D),
\end{align*} 
and then define $\bm{V}_{h}$ to be the space of vector-valued functions whose components are in $V_{h}$. 
\rbl{We will also require a compatible finite element space}
\begin{align*}
W_{h} = \textrm{span} \left\{ \varphi_{1}(\bm{x}), \varphi_{2}(\bm{x}), 
      \ldots, \varphi_{n_{p}}(\bm{x}) \right \}  \subset L^{2}(D),
\end{align*}
\rbl{in the sense that a {\it discrete} inf--sup condition 
\rrd{\begin{align}\label{discrete_inf_sup}
\sup_{0\neq\bm{v}\in {\bm{V}_{h}}}
\frac{\int_D q \,\nabla\cdot\bm{v} }{ ||\nabla\bm{v}||_{\bm{L}^2(D)} }&\ge \gamma
 ||q||_{L^2(D)} \quad\forall  q\in W_h
 \end{align}
 \rrb{is}} satisfied with $\gamma$ uniformly bounded away from zero (that is,  independent of  the mesh
parameter $h$).}
\rbl{Two specific  inf--sup stable approximation pairs are included in  our IFISS software~\cite{IFISS}
and thus \rrd{have} been extensively tested. These are $\bm{Q}_2$--${Q}_1$ \rrd{(continuous biquadratic approximation for the displacement and continuous bilinear approximation for the pressure)} and $\bm{Q}_2$--${P}_{-1}$ \rrd{(continuous biquadratic approximation for the displacement and discontinuous linear approximation for the pressure)} approximations for $\bm{V}_{h}$ and $W_{h}$ defined on a 
rectangular element subdivision}.\footnote{\rbl{Both of these mixed approximation strategies are
inf--sup stable in a three-dimensional setting.}}
 
\rbl{Turning to the parametric discretisation,}  
let $\{ \psi_{i}(y_{j}), i=0,1, \ldots \}$ denote the set of Legendre polynomials 
on $\Gamma_{j}$, where $\psi_{i}$ has degree $i$. We fix $\psi_{0}=1$ and assume that the 
polynomials are normalised in the $L_{\pi_{j}}^{2}(\Gamma_{j})$-sense, so that  
$\left<\psi_{i}, \psi_{k}\right>_{\pi_{j}} = \delta_{i,k}$. Next,  we choose a set of 
multi-indices $ \Lambda \subset \mathbb{N}_{0}^{M}$ and define the set of multivariate polynomials
\begin{align}
S_{\Lambda}: =\textrm{span}\left\{ \psi_{\boldsymbol{\alpha}}(\bm{y})
 = \prod_{i=1}^{M} \psi_{\alpha_{i}}(y_{i}), \quad \boldsymbol{\alpha} 
 \in \Lambda \right\} \subset L_{\pi}^{2}(\Gamma).
\end{align}
By construction, since $\pi$ is a product measure, the basis functions for $S_{\Lambda}$ are 
orthonormal with respect to the $L_{\pi}^{2}(\Gamma)$ inner product. We denote 
$\textrm{dim} (S_{\Lambda})  = | \Lambda |  = n_{y}$. For instance, if we choose 
$\Lambda=\left\{ \boldsymbol{\alpha} =(\alpha_{1}, \ldots, \alpha_{M}), | \boldsymbol{\alpha}| \le p  \right\}, $
then $n_{y} =\frac{(M+p)!}{M!p!}$ and $S_{\Lambda}$ contains multivariate polynomials 
of total degree $p$ or less. 

The finite-dimensional \cpbl{version of the} three-field problem \eqref{scm11a} is \cpbl{therefore}: find
 $(\rbl{\bm{u}_{h,\Lambda}}, \rbl{p_{h,\Lambda}}, \rbl{\tilde{p}_{h, \Lambda}})
  \in \bm{V}_{h, \Lambda} \times W_{h,\Lambda} \times W_{h,\Lambda}$ such that
 \begin{subequations} \label{disver11}
\begin{align}
a(\rbl{\bm{u}_{h,\Lambda}},\bm{v})+b(\bm{v},\rbl{p_{h,\Lambda}})&=f(\bm{v}) 
\quad \, \, \forall \bm{v}\in \bm{V}_{h,\Lambda},\\
b(\rbl{\bm{u}_{h,\Lambda}},q)-c(\rbl{\tilde{p}_{h, \Lambda}},q)&=0 
\quad \qquad \forall q\in  W_{h,\Lambda},\\
-c(\rbl{p_{h,\Lambda}},\tilde{q})+d(\rbl{\tilde{p}_{h, \Lambda}},\tilde{q})&=0 
\quad \qquad \forall \tilde{q}\in  W_{h,\Lambda},
\end{align}
\end{subequations}
where we define $\bm{V}_{h,\Lambda}:=\bm{V}_{h} \otimes S_{\Lambda}$ 
and $W_{h,\Lambda}:=W_{h} \otimes S_{\Lambda}$.

\rbl{The well-posedness  of the discrete formulation  follows from the  stability estimate
in Lemma~\ref{welpose11} together with the discrete inf--sup condition (\ref{discrete_inf_sup}).}
 \begin{lemma}\label{diswelpose13}
\rbl{Assuming that  $E$ satisfies (\ref{bounde11})  
and  that the  approximation pair $\bm{V}_{h}$, $W_{h}$ is inf--sup stable,}
problem (\ref{disver11}) admits a  unique solution 
$(\bm{u}_{h,\Lambda},p_{h, \Lambda},\tilde{p}_{h,\Lambda})\in\bm{{V}}_{h,\Lambda} 
\times {W}_{h,\Lambda},\times {W}_{h, \Lambda}$
satisfying 
  \begin{align}
 |||(\bm{u}_{h,\Lambda},p_{h,\Lambda},\tilde{p}_{h,\Lambda})|||\le
 \frac{C_{5}}{E_{\min}} \rbl{\alpha}^{-1/2}|| \bm{f} ||_{L^{2}(D)}, 
 \end{align}
 where $C_5$ depends on $E_{\max}$, $C_K$ and $\gamma$.
 \end{lemma}
 
 \begin{remark}
\rbl{One could, in principle, approximate $p$ and $\tilde p$ using different \cpbl{finite element} spaces. In this case
a second inf-sup condition  relating the two pressure spaces would need to be satisfied
 to ensure a stable approximation overall.}
\end{remark}

\subsection{\rbl{Linear algebra aspects}}\label{matrices}
\rbl{We will restrict our attention \rrd{to} planar elasticity from this point onwards}.\footnote{\rbl{The 
extension to three-dimensions is completely straightforward.}}
 To \rbl{formulate} the discrete linear system of equations associated with \eqref{disver11}, 
 \rbl{a set of sparse matrices and vectors associated with the chosen basis functions for the 
 approximation spaces $V_{h}, W_{h}$ and $S_{\Lambda}$ will need to be assembled.}
\rbl{To this end,  we first} define matrices 
$G_{0}, G_{k} \in \mathbb{R}^{n_{y} \times n_{y}}$ for $k=1,\ldots, M,$ by
  \begin{align*}
[G_{0}]_{\boldsymbol{\alpha},\boldsymbol{\beta}} := \int_{\Gamma}\psi_{\boldsymbol{\alpha}}(\bm{y}) 
\psi_{\boldsymbol{\beta}}(\bm{y}) \, d\pi(\bm{y}), \qquad  [G_{k}]_{\boldsymbol{\alpha},\boldsymbol{\beta}} 
:= \int_{\Gamma}  \, y_{k} \,  \psi_{\boldsymbol{\alpha}}(\bm{y}) \psi_{\boldsymbol{\beta}}(\bm{y}) \, d\pi(\bm{y}), 
  \end{align*}
where $\boldsymbol{\alpha}, \boldsymbol{\beta} \in \Lambda$.  In addition, we define 
the vector $\bm{g}_{0} \in \mathbb{R}^{n_{y}}$ to be the first column of $G_{0}$. 
 Since the basis functions for $S_{\Lambda}$ have been chosen to be orthonormal, we have 
 $G_{0}=I$. In addition, due to the three-term recurrence of the underlying univariate families of 
 Legendre polynomials, $G_{k}$ has at most two \rbl{nonzero} entries per row, for each 
 $k=1,2,\ldots, M$, see \cite{powell_elman}.

\rbl{We will} define the finite element matrix $A_{11}^{k}  \in \mathbb{R}^{n_{u} \times n_{u}}$ 
associated with $V_{h}$ by
 \begin{eqnarray*}
 [A_{11}^{k}]_{i,\ell}  : =   & \int_{D} e_{k}(\bm{x})  \,  \bm{\varepsilon}\left(
 \begin{array}{c} \phi_i(\bm{x}) \\ 0 \end{array} \right) : \bm{\varepsilon}\left(\begin{array}{c} \phi_\ell(\bm{x}) \\ 0 
 \end{array} \right) \, d\bm{x},   \quad\quad i,\ell=1, \dots, n_{u},
 \end{eqnarray*}
  for $k=0,1, \ldots, M$ and the matrix $A_{21}^{k} \in \mathbb{R}^{n_{u} \times n_{u}}$ by
 \begin{align*}
[A_{21}^{k}]_{i,\ell} := \int_{D} e_{k}(\bm{x})  \,  \bm{\varepsilon}\left(
\begin{array}{c}0 \\ \phi_i(\bm{x})  \end{array} \right) : \bm{\varepsilon}\left(\begin{array}{c} \phi_\ell(\bm{x}) \\ 0 
\end{array} \right) \, d\bm{x}, \quad\quad i,\ell=1, \dots, n_{u}.
  \end{align*}
The matrices $A_{12}^{k}, A_{22}^{k} \in \mathbb{R}^{n_{u} \times n_{u}}$ are defined analogously.  
\rbl{We can also} define matrices $B_{1}, B_{2} \in \mathbb{R}^{n_{p} \times n_{u}}$ \rbl{so that}
  \begin{align*}
  [B_{1}]_{r,\ell} = - \int_{D}   \varphi_{r}(\mathbf{x}) \,  \frac{\partial \phi_{\ell}(\mathbf{x})}{\partial x_{1}} d \bm{x} , \quad   [B_{2}]_{r,\ell} = - \int_{D}   \varphi_{r}(\mathbf{x}) \,  \frac{\partial \phi_{\ell}(\mathbf{x})}{\partial x_{2}} d \bm{x},
  \end{align*}
  for $r=1,\ldots, n_{p},  \, \ell=1, \ldots, n_{u}$. The mass matrix $C \in \mathbb{R}^{n_{p} \times n_{p}}$ associated with $W_{h}$ is defined by
    \begin{align*}
  [C]_{r,s} = \int_{D} \varphi_{r}(\mathbf{x}) \, \varphi_{s}(\mathbf{x}) \, d\mathbf{x} , \qquad r,s=1,\ldots, n_{p},
  \end{align*}
  and the weighted  \rbl{mass matrices} $D_{k} \in \mathbb{R}^{n_{p} \times n_{p}}$ are defined by
  \begin{align*}
  [D_{k}]_{r,s} =  \int_{D} e_{k}(\mathbf{x}) \varphi_{r}(\mathbf{x})\, \varphi_{s}(\mathbf{x}) \, d\mathbf{x},
   \qquad r,s=1,\ldots, n_{p},
  \end{align*}
for $k=0,1,\ldots,M$. \rbl{An important point is} that if the coefficient 
$e_{0}(\bm{x})$  \rbl{in the expansion 
of $E$ 
 is a constant} then $D_{0}= e_{0} C$. 
Moreover, if we choose $W_{h}=P_{-1}$ (discontinuous linear pressure approximation) then
 $C$ is a diagonal matrix and so is $D_{k}$, for each $k=0,1, \ldots,M$. 
 Finally, we define two vectors $\bm{f}_{1}, \bm{f}_{2} \in \mathbb{R}^{n_{u}}$ associated 
 with the body force \cpbl{$\bm{f}(\bm{x})=(f_{1}(\bm{x}), f_{2}(\bm{x}))^{\top}$}, via
 \begin{align*}
 [\bm{f}_{1}]_{\ell} = \int_{D} f_{1}(\bm{x}) \phi_{\ell}(\bm{x}) \, d \bm{x}, 
 \qquad  [\bm{f}_{2}]_{\ell} = \int_{D} f_{2}(\bm{x}) \phi_{\ell}(\bm{x}) \, d \bm{x}, \qquad \ell=1,\ldots, n_{u}.
 \end{align*}

\rbl{Permuting the variables $p_{h,\Lambda}$ and $\tilde{p}_{h,\Lambda}$ in \eqref{disver11} and 
swapping the order of the second and third equations leads to a {\it saddle-point}}
system of $2(n_{u}+n_{p})n_{y}$ \rbl{equations} 
\begin{align}
\label{saddlepoint}
\left(\begin{array}{cc} 
 \mathcal{A} & \mathcal{B}^{\top} \\ 
 \mathcal{B}  & 0  \end{array} \right)\left(\begin{array}{c} \mathbf{v} \\ \mathbf{p} \end{array} \right) = \left(\begin{array}{c} \mathbf{b} \\ \mathbf{0} \end{array} \right),
\end{align} 
where $\mathbf{b}_{1}=\mathbf{g}_{0} \otimes \mathbf{f}_{1}$, 
$\mathbf{b}_{2}=\mathbf{g}_{0} \otimes \mathbf{f}_{2}$ \rbl{with vectors}
$$ \mathbf{v}= \left(\begin{array}{c} \mathbf{u} \\ \tilde{\mathbf{p}} \end{array} \right), 
\qquad \mathbf{b}= \left(\begin{array}{c} \mathbf{b}_{1} \\ \mathbf{b}_{2} \end{array} \right),$$
\rbl{defined so that $\mathbf{u}$, $\tilde{\mathbf{p}}$ and $\mathbf{p}$  are the  coefficient vectors 
representing} $u_{h,\Lambda}$, $\tilde{p}_{h, \Lambda}$ and $p_{h, \Lambda}$, respectively
 in the chosen bases. The coefficient matrix \rbl{in (\ref{saddlepoint})} is symmetric with 
\smallskip
 \begin{align}\label{defexacta}
\mathcal{A} := & \left(\begin{array}{cc|c} 
\rbl{\alpha}  \sum_{k=0}^{M}  G_{k} \otimes A_{11}^{k}   & \rbl{\alpha} \sum_{k=0}^{M} G_{k} \otimes A_{21}^{k} & \mathbf{0} \\
& &  \\
\rbl{\alpha}  \sum_{k=0}^{M} G_{k} \otimes A_{12}^{k}  & \rbl{\alpha}  \sum_{k=0}^{M}  G_{k} \otimes A_{22}^{k}   & \mathbf{0}  \\
& &   \\ \hline 
& &   \\ 
 \mathbf{0} & \mathbf{0} & (\rbl{\alpha} \beta)^{-1}  \sum_{k=0}^{M} G_{k} \otimes D_{k} \end{array} \right)
\end{align}
and
 \begin{align}\label{defexactb}
\mathcal{B} := & \left(\begin{array}{cc|c} 
 G_{0} \otimes B_{1} & G_{0} \otimes B_{2} &   -(\rbl{\alpha} \beta)^{-1} G_{0} \otimes C
 \end{array} \right).
 \end{align}
 Note that due to its very large size, we do not assemble the full coefficient matrix. 
 Operations are only performed via the actions of $G_{0}, G_{k}$, 
 $A_{11}^{k}$, $A_{12}^{k}$, $A_{21}^{k}$, $A_{22}^{k}$, $B_{1}$, $B_{2}$, $C$, 
 and $D_{k}$.

The best way to solve a symmetric saddle-point system iteratively is to use the minimal residual method, see~\cite[Chapter 4]{HDA}. \cpbl{Since our system is ill-conditioned}, preconditioning is a critical component of our  solution strategy.

\section{Preconditioning} \label{precond-sec}
\rbl{Following~\cite{ernst2009efficient}, \cite{klawonn1998optimal}, \cite{mardal2011preconditioning}   and~\cite{ss11} the most natural preconditioner  
for  the saddle-point system (\ref{saddlepoint}) is  a block preconditioning matrix}
  \begin{align*}
P_{\textrm{approx}}= \left(\begin{array}{cc} \mathcal{A}_{\textrm{approx}} & 0 \\ 0 & \mathcal{S}_{\textrm{approx}}
 \end{array} \right),
 \end{align*}
where $\mathcal{A}_{\textrm{approx}}$ and $\mathcal{S}_{\textrm{approx}}$ are 
matrices that are chosen to \rbl{represent the matrix  $\mathcal{A}$ and the 
Schur complement
$\mathcal{S} = \mathcal{B}\mathcal{A}^{-1} \mathcal{B}^{\top}$. An important
requirement is that  the work needed to apply the action of 
 $\mathcal{A}_{\textrm{approx}}^{-1}$ and $\mathcal{S}_{\textrm{approx}}^{-1}$ is proportional to
 the dimension of  the associated approximation space.}   

 \subsection{Approximation of $\mathcal{A}$}
 \rbl{The obvious way to approximate (\ref{defexacta}) is given by}
\begin{align}\label{meanpre}
 \mathcal{A}_{\textrm{approx}}  = & \left(\begin{array}{cc|c} 
\rbl{\alpha}  \,  G_{0} \otimes {A}_{11}^{0}  & \rbl{\alpha}  \, G_{0} \otimes  {A}_{21}^{0}  & \mathbf{0} \\
& &  \\
\rbl{\alpha}  \,  G_{0} \otimes {A}_{12}^{0}   & \rbl{\alpha} \, G_{0} \otimes A_{22}^{0}   & \mathbf{0}  \\
& &   \\ \hline 
& &   \\ 
 \mathbf{0} & \mathbf{0} & (\rbl{\alpha} \beta)^{-1}  \left(G_{0} \otimes  D_{0} \right) \end{array} \right),
\end{align}
\cpbl{where we denote the diagonal blocks in (\ref{meanpre}) by
\begin{align}
\label{Aapprox2def}
\mathcal{A}_{\textrm{approx,1}}  :=  \alpha & \left(\begin{array}{cc} 
 G_{0} \otimes  {A}_{11}^{0}   &    G_{0} \otimes {A}_{21}^{0}   \\
  G_{0} \otimes  {A}_{12}^{0}   &   G_{0} \otimes  {A}_{22}^{0} 
\end{array}\right), \quad  \mathcal{A}_{\textrm{approx},2}  
:=  \frac{1}{\alpha \beta}  \left(G_{0} \otimes D_{0} \right).
\end{align}}
The fact that $G_{0}=I$  means that the nonzero terms in (\ref{meanpre}) are all block diagonal. 
Since the finite element matrices ${A}_{11}^{0}, {A}_{12}^{0},$ $ {A}_{21}^{0}, {A}_{22}^0$ 
and $D_{0}$ all involve the mean coefficient $e_{0}$, we will refer to this strategy as
 a {\it mean-based} approximation. The following lemma quantifies the  effectiveness of this approximation.

\begin{lemma}\label{approxalem11}
Let $\mathcal{A}$ and $\mathcal{A}_{\textrm{approx}}$ be defined in 
(\ref{defexacta}) and (\ref{meanpre}). \rbl{If Assumption~\ref{Assump2} holds\rrd{,}
the eigenvalues of $\mathcal{A}_{\textrm{approx}}^{-1}\, \mathcal{A}$ lie in the bounded  interval}
$\left[{E_{\min}}/{e_0^{\max}}, {E_{\max}}/{e_0^{\min}}\right]$.
\end{lemma}
\begin{proof}
The eigenvalues of $\mathcal{A}_{\textrm{approx}}^{-1}\mathcal{A}$  can be separated into two distinct sets:
each associated with one of the diagonal blocks of $\mathcal{A}$, that is
\begin{align*}
\mathcal{A}_{1} := & \rbl{\alpha} \left(\begin{array}{cc} 
  \sum_{k=0}^{M}  G_{k} \otimes A_{11}^{k}   &  \sum_{k=0}^{M} G_{k} \otimes A_{21}^{k}  \\ \smallskip
  \sum_{k=0}^{M} G_{k} \otimes A_{12}^{k}  &   \sum_{k=0}^{M}  G_{k} \otimes A_{22}^{k}  
 \end{array} \right), \quad 
\mathcal{A}_{2} :=  \frac{1}{\rbl{\alpha} \beta}  \sum_{k=0}^{M} G_{k} \otimes D_{k}.
\end{align*}
Let us consider the first block. For any $\rbl{\bm{v}\in \R^{2n_u n_y}}$ there is an associated function
$\bm{r}\in\bm{V}_{h,\Lambda}$ \rbl{and  using  (\ref{bounde11}) and (\ref{bounde11x}) gives}
  \begin{align} \nonumber
\bm{v}^\top\mathcal{A}_1\bm{v}= a(\bm{r}, \bm{r}) & = \rbl{\alpha} \int_{\Gamma} \int_{D} E(\bm{x}, \bm{y}) \,
\boldsymbol{\varepsilon}(\bm{r}): \boldsymbol{\varepsilon}(\bm{r}) d \bm{x} d\pi(\bm{y}) \\ \nonumber 
&\le\frac{E_{\max}}{e_0^{\min}} \, \rbl{\alpha} \int_{\Gamma} \int_{D} e_0(\bm{x})\, \boldsymbol{\varepsilon}(\bm{r}): 
\boldsymbol{\varepsilon}(\bm{r}) d \bm{x}  d\pi(\bm{y}) \\
& =\frac{E_{\max}}{e_0^{\min}} \, \bm{v}^\top \mathcal{A}_{\textrm{approx},1}\bm{v}. \label{aapprox11}
 \end{align}
Similarly,
  \begin{align}\label{aapprox12}
\bm{v}^\top \mathcal{A}_1\bm{v} &\ge \frac{E_{\min}}{e_0^{\max}}\,
\bm{v}^\top \mathcal{A}_{\textrm{approx},1}\bm{v}.
 \end{align}
 Combining (\ref{aapprox11}) and (\ref{aapprox12}) gives, for any $\bm{v}\neq \bm{0}$, 
  \begin{align}\label{aapprox13}
\frac{E_{\min}}{e_0^{\max}}\le\frac{\bm{v}^\top \mathcal{A}_1\bm{v} }{\bm{v}^\top 
\mathcal{A}_{\textrm{approx},1}\bm{v}}\le \frac{E_{\max}}{e_0^{\min}}.
\end{align}
\rbl{Let us consider the second block.} For any $\bm{w}\in \mathbb{R}^{n_{{p}}n_{y}}$
we can define a function $s\in W_{h, \Lambda}$ such that
 \begin{align*}
\bm{w}^\top \mathcal{A}_2\bm{w} =d(s, s) 
&= (\rbl{\alpha}\beta)^{-1} \int_{\Gamma} \int_{D} E(\bm{x}, \bm{y}) s(\bm{x},\bm{y})
     \, s(\bm{x},\bm{y})  \, d \bm{x}  d\pi(\bm{y}) \\
\bm{w}^\top \mathcal{A}_{\textrm{approx},2}\bm{w} 
&= (\rbl{\alpha}\beta)^{-1} \int_{\Gamma} \int_{D} e_{0}(\bm{x}) s(\bm{x}, \bm{y}) \, 
    s(\bm{x}, \bm{y})  \, d \bm{x}  d\pi(\bm{y}).
\end{align*}
\rbl{Making use of (\ref{bounde11}) and (\ref{bounde11x}) \rrb{again}} gives
\begin{align}\label{block2-bound}
\frac{E_{\min}}{e_0^{\max}}\le\frac{\bm{w}^\top \mathcal{A}_2\bm{w}}{\bm{w}^\top
 \mathcal{A}_{\textrm{approx},2}\bm{w}}\le \frac{E_{\max}}{e_0^{\min}}.
\end{align}
Combining the bounds for \rbl{the two} Rayleigh quotients completes the proof.
\end{proof}

 \subsection{Refined approximations of $\mathcal{A}$}
\rbl{Inverting the ${\mathcal{A}}_{\textrm{approx}, 1}$ block \cpbl{of} (\ref{meanpre}) is computationally expensive. To
address this we will look for block diagonal  alternatives of the form} 
 \begin{align}\label{cheapapprox}
 \tilde{\mathcal{A}}_{\textrm{approx},1} := \rbl{\alpha} \left(\begin{array}{cc}
G_{0} \otimes \mathbb{A}_{11}   & \bm{0}\\
 \bm{0} &   G_{0} \otimes  \mathbb{A}_{22} 
 \end{array}\right).
  \end{align}

Herein, we will consider two different  choices of $\mathbb{A}_{11}$ and $\mathbb{A}_{22}$. \cpbl{The first option is to  take $\mathbb{A}_{11}=\mathbb{A}_{22}=\mathbb{A}:=2(A_{11}^{0}+A_{22}^0)/3$. Note that, if $e_{0}=1$ then for any $\bm{v} \in \mathbb{R}^{n_{u}}$ we have $\bm{v}^{\top} \mathbb{A} \bm{v} = \| \nabla v_{h} \|_{L^{2}(D)}^{2}$ where  $v_{h}$ is the finite element function in $V_{h}$ represented by $\bm{v}$.  That is, $\mathbb{A}$ gives a discrete representation of the scalar Laplacian operator.}


 \begin{lemma}\label{lemapprox1}
Let $\mathbb{A}_{11}=\mathbb{A}_{22}=2(A_{11}^{0}+A_{22}^0)/3$. 
\rbl{If Assumption~\ref{Assump2}} holds then
all eigenvalues \rrd{$\sigma_{\!\mathcal{A}}$} of $\mathcal{A}_{\textrm{approx}}^{-1}\, \mathcal{A}$,
\rrb{where $\mathcal{A}_{\textrm{approx}}$ has leading diagonal block (\ref{cheapapprox}),}
satisfy
\begin{align} \label{eigb1}
\rrd{{\sigma_{\!\mathcal{A}} \in} }\left[  \rbl{C_K} \, \frac{E_{\min}}{e_{0}^{\max}} , \frac{E_{\max}}{e_{0}^{\min}}\right],
\end{align}
where \cpbl{$0 < C_K \le 1$} is the Korn constant. 
\end{lemma}
\begin{proof}
For any $\rbl{\bm{v} \in \mathbb{R}^{2n_{{u}}n_{y}}}$, 
we can define a function $\bm{r}\in\bm{V}_{h,\Lambda}$ such that,
\begin{align}\label{approxkorn1}
\bm{v}^\top \mathcal{A}_{1}\bm{v} &=\rbl{\alpha} \int_{\Gamma}\int_{D} E(\bm{x}, \bm{y}) 
\bm{\varepsilon}(\bm{r}): \bm{\varepsilon}(\bm{r}) d\bm{x}d\pi(\bm{y})\nonumber\\
&\le E_{\max} \, \rbl{\alpha} \int_{\Gamma}\int_{D} \nabla\bm{r}: \nabla \bm{r} d\bm{x}d\pi(\bm{y})\nonumber\\
&\le \frac{E_{\max}}{e_0^{\min}} \,  \rbl{\alpha} \int_{\Gamma}\int_{D} e_0(\bm{x}
)\nabla\bm{r}: \nabla \bm{r} d\bm{x}d\pi(\bm{y})\nonumber\\
&  = \frac{E_{\max}}{e_0^{\min}} \bm{v}^\top\tilde{\mathcal{A}}_{\textrm{approx},1}\bm{v}.
\end{align}
\rbl{Analagously,} 
\begin{align}\label{approxkorn2}
\bm{v}^\top \mathcal{A}_{1}\bm{v} 
&\ge E_{\min} \, \rbl{\alpha} \int_{\Gamma}\int_{D} 
\bm{\varepsilon}(\bm{r}): \bm{\varepsilon}(\bm{r}) d\bm{x}d\pi(\bm{y})\nonumber\\
&\ge E_{\min} \, \rbl{\alpha} C_K\int_{\Gamma}\int_{D} \nabla\bm{r}: \nabla \bm{r} d\bm{x}d\pi(\bm{y})\nonumber\\
&\ge \frac{E_{\min}}{e_0^{\max}} \, \rbl{\alpha} C_K\int_{\Gamma}\int_{D} e_0(\bm{x})
\nabla\bm{r}: \nabla \bm{r} d\bm{x}d\pi(\bm{y})\nonumber\\
& = \frac{E_{\min}}{e_0^{\max}}C_K \bm{v}^\top\tilde{\mathcal{A}}_{\textrm{approx},1}\bm{v}.
\end{align}
Combining (\ref{approxkorn1}) and (\ref{approxkorn2}) leads to bounds for the Rayleigh quotient
\begin{align*}
\frac{E_{\min}}{e_0^{\max}}C_K   \le \frac{\bm{v}^\top 
\mathcal{A}_{1}\bm{v}}{\bm{v}^\top\tilde{\mathcal{A}}_{\textrm{approx},1}\bm{v}} \le \frac{E_{\max}}{e_0^{\min}},
\end{align*}
and hence for the eigenvalues of $\tilde{\mathcal{A}}_{\textrm{approx},1}^{-1}\mathcal{A}_{1}$. 
The bound \eqref{block2-bound} provides a bound for the eigenvalues  of 
$\rrb{{\mathcal{A}}}_{\textrm{approx},2}^{-1}\mathcal{A}_{2}$. 
Combining these \rbl{two bounds}  gives the \rbl{stated} result.
\end{proof}

\rbl{For our second choice of  $\mathbb{A}_{11}$ and $\mathbb{A}_{22}$}, we simply discard the 
off-diagonal blocks of the mean-based approximation ${\mathcal A}_{\textrm{approx},1}$.
\rbl{The strategy will not be pursued here since it results in an inferior eigenvalue bound.}
 \begin{lemma}\label{lemapprox2}
Let $\mathbb{A}_{11}=A_{11}^{0}$ and $\mathbb{A}_{22}=A_{22}^0$. 
\rbl{If Assumption~\ref{Assump2}} holds then
all eigenvalues \rrd{$\sigma_{\!\mathcal{A}}$} of $\mathcal{A}_{\textrm{approx}}^{-1}\, \mathcal{A}$,
\rrb{where $\mathcal{A}_{\textrm{approx}}$ has leading diagonal block (\ref{cheapapprox}),}
satisfy
\begin{align} \label{eigb2}
\rrd{\sigma_{\!\mathcal{A}}}\in \left[ C_{K} \, \frac{E_{\min}}{e_{0}^{\max}}, 2\frac{E_{\max}}{e_{0}^{\min}}\right],
\end{align}
where \cpbl{$0 < C_K \le 1$} is the Korn constant. 
\end{lemma}
\begin{proof}
 The proof is a minor variation of that of Lemma~\ref{lemapprox1}. 
 By obtaining bounds for both Rayleigh quotients separately, we find that the eigenvalues lie in 
 $$\left[C_{K}\frac{E_{\min}}{e_{0}^{\max}}, 2\frac{E_{\max}}{e_{0}^{\min}}\right] \cup 
 \left[\frac{E_{\min}}{e_{0}^{\max}}, \frac{E_{\max}}{e_{0}^{\min}}\right],$$
which yields the stated result. 
\end{proof}

\begin{remark}
The bounds \eqref{eigb1} and \eqref{eigb2} \rbl{depend} on the Young's modulus $E$ 
and on the Korn constant $C_{K}$ but are independent of all discretisation parameters. 
\end{remark}

 \subsection{\rbl{Approximation of $\mathcal{S}$}}
 Given \rbl{a block diagonal approximation to $\rrd{{\mathcal{A}}_1}$  of the form 
 \eqref{cheapapprox}, an approximation to the Schur complement matrix 
 $\mathcal{S}$ can be constructed so that}  
 $\tilde{\mathcal{S}}_{\textrm{approx}}:=\mathcal{B}\tilde{\mathcal{A}}_{\textrm{approx}}^{-1}\mathcal{B}^{\top}$. 
 \rbl{Since  this is a dense matrix it is not a practical  preconditioner.} 
 The next result \rbl{introduces a sparse block diagonal matrix} $P_{\mathcal{S}}$ and establishes that it is 
 spectrally equivalent to $\tilde{\mathcal{S}}_{\textrm{approx}}.$
 
 \smallskip 
  \begin{lemma}\label{lemapproxshurboud11}
 \rbl{Suppose that} $\tilde{\mathcal{S}}_{\textrm{approx}}:=
 \mathcal{B}\tilde{\mathcal{A}}_{\textrm{approx}}^{-1}\mathcal{B}^{\top}$ 
 where, in the definition \rbl{\eqref{cheapapprox} 
 we make the choice} $\mathbb{A}_{11}=\mathbb{A}_{22}=2(A_{11}^0+A_{22}^0)/3$. 
 \rbl{Defining} 
 \begin{align}\label{PS-def}
 P_S:=(\rbl{\alpha}^{-1}+(\rbl{\alpha}\beta)^{-1})\, I \otimes C,
 \end{align}
where $C$ is \rbl{the pressure mass matrix, we have}
 \begin{align}\label{PS-bound1}
 \theta^{2} \le \frac{{\bm{w}}^\top\tilde{\mathcal{S}}_{\textrm{approx}}{\bm{w}}}{{\bm{w}}^\top P_{\mathcal{S}}{\bm{w}}}\le  \Theta^{2} \qquad\forall \bm{w}\in\mathbb{R}^{n_p n_{y}} ,
 \end{align}
\rbl{with $\theta^{2}= {\gamma^2}/{e_0^{\max}}$, $\Theta^{2}=  {2}/{e_0^{\min}}$, where}
$\gamma$ is the discrete inf--sup constant \rrb{in (\ref{discrete_inf_sup})} associated with the finite element 
spaces $\bm{V}_{h}$ and $W_{h}$.
 \end{lemma}
 \begin{proof}
 Using the definitions of $\mathcal{B}$ and $\tilde{\mathcal{A}}_{\textrm{approx}}$ 
 and the fact that $G_{0}=I$ gives
 \begin{align}\label{meantildeshurapprox}
 \tilde{\mathcal{S}}_{\textrm{approx}} &=  \left(I \otimes B_{1}\right) \left(\rbl{\alpha}   
 \left(I \otimes \mathbb{A}_{11} \right)\right)^{-1}  \left(I \otimes B_{1}\right)^{\top} 
+   \left( I \otimes B_{2} \right) \left( \rbl{\alpha} \left(I \otimes \mathbb{A}_{22} \right) \right)^{-1} 
\left( I \otimes B_{2} \right)^{\top}  \nonumber\\
&\qquad +  \left(-(\rbl{\alpha} \beta)^{-1} I \otimes C\right) \left((\rbl{\alpha} \beta)^{-1} 
    \left(I \otimes  D_{0}\right) \right)^{-1} 
 \left(-(\rbl{\alpha} \beta)^{-1} I \otimes C\right)^{\top}\nonumber \\
 &= \rbl{\alpha}^{-1} \left( I \otimes \left(B_{1}\mathbb{A}_{11}^{-1}B_{1}^{\top}\right) 
 +  I \otimes \left(B_{2}\mathbb{A}_{22}^{-1}B_{2}^{\top}\right)\right)  
 + \left(\rbl{\alpha} \beta\right)^{-1} \left(I \otimes C D_{0}^{-1}C^{\top}\right)\nonumber \\
&= {\rbl{\alpha}^{-1} }\left( I \otimes X\right) +  { \left(\rbl{\alpha} \beta\right)^{-1}} 
\left(I \otimes C D_{0}^{-1}C^{\top}\right),
 \end{align}
 where $X:=\left(B_{1}\mathbb{A}_{11}^{-1}B_{1}^{\top} + B_{2}\mathbb{A}_{22}^{-1}B_{2}^{\top}\right)$. 
 
The fact that the matrices $\mathbb{A}_{11}$ and $ \mathbb{A}_{22}$ represent  discrete Laplacian operators weighted by the mean field $e_0$, 
gives the  matrix $X$ a structure that can be exploited. Specifically  we can combine the bounds in \cite[Proposition 3.24]{HDA} with the bounds on 
$e_0$ in (\ref{bounde11x}) to give a two-sided bound
 \begin{align*}
  \frac{\gamma^{2}}{e_{0}^{\max}} \le \frac{\bm{w}^{\top} \left(I \otimes X \right)
   \bm{w}}{\bm{w}^{\top} \left(I \otimes C \right) \bm{w}} \le  \frac{2}{e_{0}^{\min}}
  \qquad\forall \bm{w}\in\mathbb{R}^{n_p n_{y}},
 \end{align*}
 where  \rbl{$\gamma$ is the inf--sup constant (as defined in (\ref{discrete_inf_sup})}). 
 We also have \rbl{a  two-sided bound for the two component mass matrices}
 \begin{align}\label{capprox1}
 e_{0}^{\min} \, C \le D_{0}\le e_{0}^{\max} \,C ,
 \end{align}
 where the \rbl{inequalities} hold entrywise. 
 Combining these results with \eqref{meantildeshurapprox} gives
 \begin{align*}
 {\bm{w}}^\top \tilde{\mathcal{S}}_{\textrm{approx}}{\bm{w}}  
 &=  \rbl{\alpha}^{-1} {\bm{w}}^\top \left( I\otimes X \right) {\bm{w}}
 +(\rbl{\alpha}\beta)^{-1} {\bm{w}}^\top \left( I\otimes C D_{0}^{-1}C^{\top} \right) {\bm{w}}, \\
 &\le \rrd{2} (\rbl{\alpha} e_0^{\min})^{-1} {\bm{w}}^\top \left( I\otimes C \right){\bm{w}}
 +(\rbl{\alpha}\beta e_{0}^{\min})^{-1} {\bm{w}}^\top \left( I\otimes C \right) {\bm{w}} ,\nonumber\\
 & \rrd{\le2} (e_0^{\min})^{-1} {\bm{w}}^\top P_{\mathcal{S}} {\bm{w}}. 
 \end{align*}
 Similarly, 
 \begin{align}\label{newschurapprox11}
 {\bm{w}}^\top\tilde{\mathcal{S}}_{\textrm{approx}}{\bm{w}} \ge &  (\rbl{\alpha} e_0^{\max})^{-1}\gamma^2 {\bm{w}}^\top \left( I\otimes C \right) {\bm{w}}+(\rbl{\alpha}\beta e_0^{\max})^{-1} {\bm{w}}^\top I\otimes C {\bm{w}}, \nonumber\\
 & =  \gamma^2 ({e_0^{\max}})^{-1} {\bm{w}}^\top P_{\mathcal{S}} {\bm{w}}.
 \end{align}
 Combining the upper and lower bounds leads to the stated result.
 \end{proof}
 
\begin{remark}\label{refinedest}
In a practical setting, the mean field $e_0$ in (\ref{E-def}) is often taken to be constant. 
In this case we could define
$P_{\mathcal{S}}:=(\rbl{\alpha}^{-1}+(\rbl{\alpha}\beta)^{-1})\, e_0^{-1} I\otimes C$
and get a refined estimate
  \begin{align} \label{PS-bound2}
\theta^{2}:= \gamma^2\le \frac{{\bm{w}}^\top\tilde{\mathcal{S}}_{approx}{\bm{w}}}{{\bm{w}}^\top P_{\mathcal{S}}{\bm{w}}}\le 2 :=\Theta^{2}\qquad\forall \bm{w}\in\mathbb{R}^{n_p n_{y}}.
\end{align}
\end{remark}
Notice that $P_{S}$ is block diagonal but if we choose $W_{h}=P_{-1}$ then $C$ is diagonal and hence so is $P_{S}$.  
 
We will summarise our preferred  methodology  at this point: the preconditioner of choice  is a block diagonal matrix
 \begin{align}\label{preapprox11}
 \rbl{{ \cal P} := \left(\begin{array}{ccc} 
 \tilde{\mathcal{A}}_{\textrm{approx,1}} & 0 &0 \\ 
0 &  {\mathcal{A}}_{\textrm{approx,2}} & 0  \\ 
0 & 0 & P_{\mathcal{S}}  \end{array} \right) , }
\end{align}
where $\rrd{\tilde{\mathcal{A}}}_{\textrm{approx},1}$ \rbl{is as defined
in} \eqref{cheapapprox} with $\mathbb{A}_{11}=\mathbb{A}_{22}=2(A_{11}^0+A_{22}^0)/3$,
\rbl{${\mathcal{A}}_{\textrm{approx},2}$ is defined in  \eqref{Aapprox2def}} and 
$P_S$ is defined in \eqref{PS-def} (or else as in \eqref{PS-bound2} if $e_{0}$ is constant). 

We note that each of the three diagonal blocks of the preconditioner, $ \tilde{\mathcal{A}}_{\textrm{approx},1}$,  $\mathcal{A}_{\textrm{approx},2}$ and $P_{\mathcal{S}}$ provides a discrete representation of a norm that is equivalent to one of the terms in the norm (\ref{normden}). This strategy is consistent with the preconditioning philosophy of Mardal \& Winther~\cite{mardal2011preconditioning} and ensures that the eigenvalues of the preconditioned system can be bounded independently of the discretisation parameters. This is formally expressed in the following concluding result.

 \begin{theorem}\label{thesgfem11} \cpbl{Let ${\mu}_{\min}$ and ${\mu}_{\max}$ be} the extremal eigenvalues 
of $\rrb{{\mathcal{A}}}_{approx}^{-1}\mathcal{A}$, where $\mathcal{A}_{\textrm{approx}}$ has leading diagonal block (\ref{cheapapprox}) with $\mathbb{A}_{11}=\mathbb{A}_{22}=2(A_{11}^0+A_{22}^0)/3$. Then the eigenvalues of
\begin{align}
\rbl{{\cal P}}^{-1/2} \left(\begin{array}{cc} 
 \mathcal{A} & \mathcal{B}^{\top} \\ 
 \mathcal{B}  & 0  \end{array} \right)  \rbl{{\cal P}}^{-1/2},
 \end{align}
\rbl{lie in the union of the intervals}
\begin{align}
\label{finalbound}
\Bigg[\frac{1}{2}\left({\mu}_{\min}-\sqrt{{\mu}_{\min}^2
+4{\Theta}^2}\right),&\frac{1}{2} \left({\mu}_{\max}-\sqrt{{\mu}_{\max}^2+4{\theta}^2}\right) \Bigg]\nonumber\\
& \cup\left[{\mu}_{\min},\frac{1} {2}\left({\mu}_{\max}^2
+\sqrt{{\mu}_{\max}+4{\Theta}^2}\right)\right],
\end{align}
where the constants $\theta$ and $\Theta$ are  given in 
Lemma~\ref{lemapproxshurboud11} if $P_{\mathcal{S}}$ \rbl{is as defined in}
 \eqref{PS-def}, or else are given in \eqref{PS-bound2} if $P_{\mathcal{S}}$ \rbl{has the alternative definition
given in  Remark~\ref{refinedest}}.
 \end{theorem}
 \begin{proof}
The proof follows from Lemma 2.1 of \cite{rusten1992preconditioned} and Corollary  3.4 of \cite{powell-PhD}. 
 \end{proof}
 
 \smallskip 
 
Recall that bounds for the eigenvalues of $\tilde{\mathcal{A}}_{approx}^{-1}\mathcal{A}$ are given 
in \eqref{eigb1}. Hence, the bounds for the eigenvalues for the preconditioned system  depend {\it only}
on the  discrete inf-sup constant $\gamma$ \rrb{in (\ref{discrete_inf_sup})}, the Korn constant $C_{K}$ and 
 the ratios $E_{\min} / e_{0}^{\max}$ and $E_{\max} / e_{0}^{\min}$\rrd{.} 
\rbl{Note that the eigenvalue bounds are robust in the incompressible limit.} 
\rbl{A direct consequence of our eigenvalue bound is that the number of MINRES iterations
needed to converge to a fixed tolerance when solving the Galerkin system is guaranteed to be  
bounded by a constant that is independent  of all discretisation parameters as well as the Poisson ratio.
This will be illustrated by numerical results in the final section.}
 
  \section{Numerical results}\label{Numer}
   
In this section we  consider a representative test problem taken from the S-IFISS toolbox \cite{SIFISS} and we study the  practical performance of the block diagonal preconditioning strategy that was analysed above.  The spatial domain is $D=(-1,1)\times(-1,1)$. We impose a homogeneous Neumann boundary condition on the right edge $\partial D_N=\{1\}\times(-1,1)$ and a zero essential boundary condition for the displacement on $\partial D_D=\partial D\setminus\partial D_N$.  The  body force is  chosen to be $\bm{f}=(1,1)^{\top}$.  The Young's modulus has \rbl{constant} mean value one and \rbl{takes} the form
\begin{align}
E(\bm{x},\bm{y})= 1 +  \sigma  \sqrt{3} \sum_{m=1}^{M}\sqrt{\lambda_m}\varphi_m(\bm{x})y_m,
\end{align}
where $\sigma$ is the standard deviation and $\{(\lambda_m,\varphi_m)\}$ are the eigenpairs 
of the integral operator associated with $\cpbl{(1/\sigma^{2})}C(\bm{x}, \bm{x}')$, where 
\begin{align}
C(\bm{x},\bm{x}')=\sigma^2 \exp\left(-\frac{1}{2} ||\bm{x}-\bm{x}' ||_{1} \right), \quad\bm{x},\bm{x}'\in D.
\end{align}
For the spatial approximation, we use $\bm{Q}_2-\rrd{P}_{-1}-\rrd{P}_{-1}$ mixed finite elements. That is, continuous biquadratic approximation for the displacement and discontinuous linear approximation for both of the Lagrange multipliers. In this case, the approximation $P_{\mathcal{S}}$ to the Schur complement is a diagonal matrix. For the parametric approximation, we choose $S_{\Lambda}$ to be the set of polynomials of total degree $p$ or less in $y_1,\ldots,y_M$ on $\Gamma=[-1,1]^{M}$.  In Table \ref{det_DOF} we record the number of spatial degrees of freedom associated with the finite element discretisation (as the refinement level $\ell$ is varied) and in Table \ref{stoch_DOF} we record the dimension of the space $S_{\Lambda}$ (when $M$ and $p$ are varied). Recall that the number of equations to be solved is $2(n_{u}+n_{p})n_{y}$. For example, when we have $M=10$ input parameters, the grid level is set to $\ell=6$ and the polynomial degree is $p=4$, we have over fourteen million equations to solve.

\begin{table}[ht!]
 \caption[]{Number of deterministic degrees of freedom associated with $\bm{Q}_2-P_{-1}-P_{-1}$ approximation.}
  \label{det_DOF}
 \begin{center}
    \begin{tabular}{ | c | c | c | c | c | c | c | c | c | c |c| }
    \hline
    \multicolumn{4}{| c |}{deterministic degrees of freedom }\\
       \hline   
    \multicolumn{1}{| c |}{Refinement-level ($l$)}& $n_u$&$n_p$ &$2(n_u+n_p)$  \\
    \hline
    $4$&$240$&$192$&$864$\\
     $5$&$992$&$768$&$3,520$\\
    $6$&$4,032$&$3,072$&$14,208$\\
    \hline
  \end{tabular}   
\end{center}
\end{table}

\begin{table}[ht!]
 \caption[]{Number of parametric degrees of freedom associated with the chosen multi-index set $\Lambda$.}
 \label{stoch_DOF}
 \begin{center}
    \begin{tabular}{ | c | c | c | c | c | c | c | c | c | c |c| }
    \hline
         \multicolumn{4}{| c |}{$n_y$}\\
       \hline   
    \multicolumn{1}{| c |}{$p$}& $M=5$&$M=8$ &$M=10$  \\
    \hline
    $3$&$56$&$165$&$286$\\
    $4$&$126$&$495$& $1,001$ \\
    \hline
  \end{tabular}   
\end{center}
\end{table}

We examine the eigenvalues of the preconditioned SGFEM system first. \rbl{The {\tt est\_minres} code that is built into S-IFISS exploits the 
connection with the Lanczos algorithm (see \cite[section~2.4]{HDA})  and generates accurate harmonic Ritz values estimates of the underlying 
eigenvalue spectrum as the preconditioned system is being solved. Details are given in  Silvester \& Simoncini~\cite{ss11}. The extremal eigenvalue 
estimates are computed on the fly and are reproduced in Tables \ref{eigT1} and \ref{eigT2}.}\footnote{\rbl{The associated MINRES 
relative residual  tolerance is set to $10^{-6}$. \rrd{Bounds} are unchanged if we rerun  the experiments with a tighter  tolerance.}}
We consider two values of the Poisson ratio $\nu$, and two values for the standard deviation $\sigma$ (values which guarantee that all realisations of $E$ are positive) 
and vary $M$ and $\ell$.  The polynomial degree $p=3$ is fixed. We observe that the widths of the intervals containing the estimated eigenvalues 
are independent of the spatial discretisation parameter as well as the \cpbl{number $M$} of parameters. While the intervals are slightly wider for $\nu=0.49999$ than for 
$\nu=0.4$, they are bounded as $\nu \to 1/2$.  \cpbl{The interior eigenvalue bounds are closer to the origin, however, for the larger value of $\sigma$. The price we pay for using a \emph{mean-based} preconditioner (which is by definition block diagonal) is that the resulting eigenvalue bounds depend on the ratios $E_{\min}/e_{0}^{\max}$ and $E_{\max}/e_{0}^{\min}$. In this example, these quantities depend on $\sigma$. However, we stress that $\sigma$ cannot be chosen arbitrarily large. Assumption \ref{Assump2} must be satisfied, otherwise the problem is not well posed. Preconditioning schemes that are not mean-based may lead to more tightly clustered eigenvalues but in general are not as computationally efficient.}

\begin{table}[ht!]
 \caption[]{\rrd{Bound for} eigenvalues of preconditioned SGFEM system, $\sigma =0.085$, $p=3$.}
 \label{eigT1}
\begin{center}
    \begin{tabular}{ | c | c | c | c | c | c | c | c | c | c |c| }
    \hline
    \multicolumn{3}{| c |}{Computed eigenvalue }\\
      \hline
     \multicolumn{3}{| c |}{$l=5$}\\
       \hline   
    \multicolumn{1}{| c |}{$M$}&{$\nu=.4$ }& $\nu=.49999$  \\
    \hline
    $5$&$[-0.8287, -0.3369]\cup [ 0.2737, 1.8332]$&$[-0.9347, -0.1892]\cup [0.2878, 1.8886]$\\
        $8$&$[-0.8305, -0.3368]\cup[0.2722, 1.8408]$&$[-0.9058, -0.1891]\cup[0.2859, 1.8934]$\\
          $10$&$[-0.8311, -0.3367]\cup[0.2720, 1.8427]$&$[-0.9064, -0.1891]\cup[ 0.2857, 1.8949]$\\
              \hline
    \multicolumn{3}{| c |}{$l=6$}\\
            \hline
      $5$&$[-0.8291, -0.3368 ]\cup [0.2731, 1.8358]$&$[-0.9047, -0.1890]\cup [0.2866, 1.8910]$\\
          $8$&$[-0.8323, -0.3366]\cup[0.2715, 1.8448]$&$[-0.9084, -0.1890]\cup[0.2849, 1.8986]$\\    
          $10$&$[-0.8334, -0.3366]\cup[0.2713, 1.8469]$&$[-0.9094, -0.1890]\cup[0.2848, 1.9006]$\\
              \hline
  \end{tabular}   
\end{center}
\end{table}

\begin{table}[ht!] 
 \caption[]{\rrd{Bound for} eigenvalues of preconditioned SGFEM system, $\sigma =0.17$, $p=3$.}
 \label{eigT2}
  \begin{center}
    \begin{tabular}{ | c | c | c | c | c | c | c | c | c | c |c| }
    \hline
    \multicolumn{3}{| c |}{Computed eigenvalue }\\
      \hline
     \multicolumn{3}{| c |}{$l=5$}\\
       \hline   
    \multicolumn{1}{| c |}{$M$}&{$\nu=.4$ }& $\nu=.49999$  \\
    \hline
    $5$&$[-0.9291, -0.3178]\cup [0.2318, 1.9435]$&$[-0.9491, -0.1789]\cup [0.2428, 1.9935]$\\
        $8$&$[-0.8797, -0.3171]\cup[0.2268, 1.9566]$&$[-0.9538 , -0.1789]\cup[0.2358, 2.0052]$\\
          $10$&$[-0.8817, -0.3169]\cup[0.2264, 1.9604]$&$[-0.9555, -0.1788]\cup[ 0.2352, 2.0086]$\\
              \hline
    \multicolumn{3}{| c |}{$l=6$}\\
            \hline
      $5$&$[-0.9206, -0.3176 ]\cup [0.2307, 1.9454]$&$[-0.9507, -0.1787]\cup [0.2413, 1.9964]$\\
          $8$&$[-0.8836, -0.3167]\cup[0.2254, 1.9623]$&$[-0.9581, -0.1787]\cup[0.2346, 2.0126]$\\    
          $10$&$[-0.8857, -0.3166]\cup[0.2251, 1.9663]$&$[-0.9600, -0.1785]\cup[0.2336, 2.0167]$\\
              \hline
  \end{tabular}   
\end{center}
\end{table}

In Table \ref{small_sigma_p3} we record the number of MINRES iterations required to reduce the 
\rbl{preconditioned} residual error to $10^{-6}$ for the case $\sigma=0.085$, with $p=3$ fixed and 
varying $M$  and $\ell$. In Tables \ref{large_sigma_p3} and \ref{large_sigma_p4} we record 
the number of iterations required when $\sigma=0.17$ with $p=3$ and $p=4$ fixed, respectively.
The timings were recorded running S-IFISS on  a MacBook Pro with 16Gb of memory
and a 2.3GHz  Intel Core i5 processor. We observe that  for a fixed value of $\sigma$, the iteration counts remain stable as the 
discretisation parameters $\ell$ and $p$ are varied.  \rbl{Moreover, the iteration counts stay 
bounded when working with values of $\nu$  arbitrarily close to $1/2$.}

\begin{table}[ht!]
 \caption[]{MINRES iteration counts for stopping tolerance $10^{-6}$, and timings in seconds (in parentheses), $\sigma =0.085$, $p=3.$  }
    \label{small_sigma_p3}
\begin{center}
    \begin{tabular}{ | l | l | l | l | l | l | l | l | l | l | l | }
      \hline   
    \multicolumn{1}{| c |}{$M$}&{$\nu=.4$ }& $\nu=.49$&$\nu=.499$&$\nu=.4999$&$\nu=.49999$  \\
    \hline
     \multicolumn{6}{| c |}{$l=5$}\\
    \hline
    $5$&$56 (3.9)$&$74 (5.3)$&$78 (5.7)$&$78 (5.7)$&$78 (5.8)$\\
      $8$&$56 (10)$&$75 (12.8)$&$78 (13.7)$&$79 (13.7)$&$79 (13.4)$\\
      $10$&$56 (16.5)$&$75 (22.5)$&$79 (23.6)$&$79 (23.4)$&$79 (23.5)$\\
              \hline
    \multicolumn{6}{| c |}{$l=6$}\\
            \hline
      $5$&$56(14.6)$&$75(19.7)$&$79(20.9)$&$79(20.5)$&$79(20.9)$\\
          $8$&$56(45.2)$&$75(60.5)$&$79(64.2)$&$79(63.8)$&$79(64)$\\   
           $10$&$56(86)$&$75(114.3)$&$79(120.5)$&$79(118.1)$&$79(117.3)$\\
              \hline
  \end{tabular}   
\end{center}
\end{table}

\begin{table}[ht!]
 \caption[]{MINRES iteration counts for stopping tolerance  $10^{-6}$, and timings in seconds (in parentheses), $\sigma =0.17$, $p=3$.  }
    \label{large_sigma_p3}
\begin{center}
    \begin{tabular}{ | l | l | l | l | l | l | l | l | l | l | l | }
       \hline   
    \multicolumn{1}{| c |}{$M$}&{$\nu=.4$ }& $\nu=.49$&$\nu=.499$&$\nu=.4999$&$\nu=.49999$  \\
    \hline
     \multicolumn{6}{| c |}{$l=5$}\\
    \hline
    $5$&$66(5.4)$&$86(6.3)$&$90(6.8)$&$92(6.9)$&$92(6.9)$\\
      $8$&$67(11.5)$&$88(15.4)$&$92(15.8)$&$93(16.4)$&$93(15.9)$\\
      $10$&$67(19.9)$&$88(27)$&$93(28.4)$&$93(28.1)$&$93(28.6)$\\
              \hline
    \multicolumn{6}{| c |}{$l=6$}\\
            \hline
      $5$&$66(18.3)$&$88(24.4)$&$92(25.5)$&$92(25.5)$&$92(25.5)$\\
          $8$&$67(55.4)$&$88(70.7)$&$93(75.8)$&$93(75.4)$&$93(76.6)$\\   
           $10$&$67(102.4)$&$89(134.9)$&$93(140.4)$&$95(145)$&$95(142)$\\
              \hline
  \end{tabular}   
\end{center}

\end{table}

\begin{table}[ht!]
 \caption[]{MINRES iteration counts for stopping tolerance $10^{-6}$, and timings in seconds (in parentheses),  $\sigma =0.17$, $p=4$.  }
    \label{large_sigma_p4}
\begin{center}
    \begin{tabular}{ | l | l | l | l | l | l | l | l | l | l | l |  }
        \hline   
    \multicolumn{1}{| c |}{$M$}&{$\nu=.4$ }& $\nu=.49$&$\nu=.499$&$\nu=.4999$&$\nu=.49999$  \\
    \hline
     \multicolumn{6}{| c |}{$l=5$}\\
    \hline
    $5$&$67(8.1)$&$90(11.6)$&$95(12)$&$95(12)$&$95(12)$\\
      $8$&$70(34.4)$&$93(44.4)$&$97(45.6)$&$98(48.9)$&$98(48.1)$\\
      $10$&$70(69.4)$&$93(94.1)$&$98(96.7)$&$98(96.5)$&$98(94.2)$\\
              \hline
    \multicolumn{6}{| c |}{$l=6$}\\
            \hline
      $5$&$69(39.9)$&$91(50.6)$&$95(54.7)$&$96(54.6)$&$96(53.6)$\\
          $8$&$70(176.8)$&$94(233.3)$&$98(249.4)$&$98(249.2)$&$98(250.6)$\\   
           $10$&$70(378.2)$&$94(513.9)$&$98(538.1)$&$98(538.7)$&$98(534.3)$\\
              \hline
  \end{tabular}   
\end{center}
\end{table}

\newpage 

\section{Conclusions}  
This work analyses  parameter-robust  discretizations and the construction of preconditioners for linear elasticity problems with uncertain material parameters. 
Having introduced a new three-field formulation of the problem, it is rigorously shown that preconditioners that are based on mapping properties associated with a specific 
parameter-dependent norm are robust with respect to variations of the \cpbl{Poisson ratio},  the choice of finite elements spaces  as well as the discretization parameters.  The theoretical results are confirmed by a  systematically designed set of numerical experiments. \cpbl{There are several generalizable aspects of our work. The idea of introducing an additional auxiliary variable to avoid working with parameter-dependent coefficients that exhibit rational nonlinearities is applicable to other PDE problems. The ideas underpinning the construction of the block diagonal preconditioner apply to other PDEs  with parameter-dependent coefficients where stochastic Galerkin approximation leads to a discrete problem with the structure (3.5). Finally, the preconditioning methodology gives a starting point for designing efficient solution algorithms  for more challenging (e.g., nonlinear) elasticity models with uncertain material coefficients.}


\bibliographystyle{siam}
\bibliography{kps_upd}
\end{document}